\documentclass[11pt,english,a4paper]{smfart}

\theoremstyle{theorem}
\newtheorem{theorem}{Theorem}

\newtheorem{lem}[theorem]{Lemma}
\newtheorem{cor}[theorem]{Corollary}
\newtheorem*{theov}{Vague Theorem}

\theoremstyle{definition}

\usepackage[french,main=english]{babel}
\usepackage[utf8, latin1]{inputenc}

\usepackage{graphicx}
\usepackage{ae,amsfonts,euscript,enumerate}
\usepackage{amssymb}
\usepackage{amsmath}
\usepackage[cm]{aeguill}
\usepackage{xcolor}

\renewcommand{\theequation}{\thesection.\arabic{equation}}

\newcommand{\bZ}{\mathbb{Z}}
\newcommand{\bQ}{\mathbb{Q}}

\newcommand{\Qbar}{\overline{\mathbb{Q}}}
\newcommand{\bC}{\mathbb{C}}

\newcommand{\cF}{\mathcal{F}}
\newcommand{\cG}{\mathcal{G}}
\newcommand{\kI}{\mathfrak{I}}

\newcommand{\Alg}{\mathrm{Alg}_{\overline{\mathbb{Q}}}}
\newcommand{\cE}{\mathcal{E}}
\newcommand{\Span}{\mathrm{span}}
\newcommand{\ev}{\mathrm{ev}}
\newcommand{\degtr}{\mathrm{degtr}}
\newcommand{\rmht}{\mathrm{ht}}

\DeclareMathOperator{\lcm}{lcm}

\author{B. Adamczewski}
\address{
Universit\'e Claude Bernard Lyon 1, CNRS, \'Ecole Centrale de Lyon, INSA Lyon, Universit\'e Jean Monnet, ICJ UMR5208, 69622 Villeurbanne, France.}
\email{Boris.Adamczewski@math.cnrs.fr}

\author{\'E. Delaygue}
\address{
Universit\'e Claude Bernard Lyon 1, CNRS, \'Ecole Centrale de Lyon, INSA Lyon, Universit\'e Jean Monnet, ICJ UMR5208, 69622 Villeurbanne, France.}
\email{delaygue@math.univ-lyon1.fr}

\title{Algebraic relations between sine and cosine values}
\date{}

\thanks{This project has received funding from the ANR project De Rerum Natura (ANR-19-CE40-0018).} 

\subjclass{11J81, 11J85}

\keywords{Transcendence, algebraic independence, Lindemann--Weierstrass theorem, trigonometric functions, Siegel $E$-functions}

\begin{document}

\maketitle

\maketitle

\begin{abstract}
The aim of this note is to show that any algebraic relation over $\Qbar$ between the values of the trigonometric functions sine and cosine at algebraic points can be derived from the Pythagorean identity and the angle addition formulas. This result is obtained as a consequence of the Lindemann--Weierstrass theorem. 
\end{abstract}

\section{Introduction.}

In their trigonometry course, all high school students learn the Pythagorean identity
\begin{equation}\label{eq: pyt}
\cos^2\alpha +\sin^2\alpha =1 \,,
\end{equation}
as well as the angle addition formulas
\begin{equation}\label{eq: add}
\begin{cases}
\cos(\alpha+\beta)=\cos\alpha\cos\beta -\sin\alpha\sin\beta\,,\\
 \sin(\alpha+\beta)=\sin\alpha\cos\beta+\cos\alpha\sin\beta\,.
\end{cases}
\end{equation}
These fundamental identities hold for \emph{all} complex numbers $\alpha$ and $\beta$. Of course, they can be used recursively to produce various polynomial relations between the values of the trigonometric functions sine and cosine, such as 
$$
\cos(2\alpha)\cos^2(\alpha)+ 4 \sin(2\alpha)\sin \left(\frac{\alpha}{2}\right)\cos \left(\frac{\alpha}{2}\right)\cos(\alpha)-\sin^2(\alpha)\cos(2\alpha) =1\,,
$$
which is valid for all complex numbers $\alpha$. The aim of this note is to study the following question: 

\medskip

\emph{Is it true that any polynomial relation with algebraic coefficients between sine and cosine values can be derived from the fundamental geometric identities \eqref{eq: pyt} and \eqref{eq: add}?}

\medskip

In general, the answer is no. Indeed, there are some sporadic relations such as
$$
\sin(\pi)=0 \quad \mbox{ and }\quad \cos^2 \left(\frac{\pi}{5}\right) - \frac{1}{2} \cos\left(\frac{\pi}{5}\right) -\frac{1}{4}=0 \,,
$$
that cannot be obtained in this way. The reason is that, if they could, they would remain valid by replacing the transcendental numbers $\pi$ and $\pi/5$ respectively by any complex number $\alpha$, which is obviously not the case. 
However, as the following result shows, the answer to our question turns out to be positive if we restrict ourselves to the values of sine and cosine at \emph{algebraic points}. In the sequel, we fix an embedding of the field of algebraic numbers $\Qbar$ into $\bC$. 

\begin{theov}\label{theo: main0}
Any algebraic relation over $\Qbar$ between the values of sine and cosine at algebraic points can be derived from the Pythagorean identity and the 
angle addition formulas. 
\end{theov}

Since the expression ``can be derived from the Pythagorean identity and the 
angle addition formulas" is somewhat imprecise, our first task is to formalize our \emph{vague theorem} into a precise one which expresses the very same idea (cf.\ Theorem \ref{theo: main}). 
In order to do that, we introduce the $\Qbar$-algebra $\mathcal T$ 
formed by all polynomial expressions in algebraic numbers and the values of sine and cosine at algebraic points, that is  
$$
\mathcal T := \Qbar[\cos(\alpha),\sin(\alpha) : \alpha\in\Qbar]\subseteq \mathbb C\,.
$$
Then we introduce the ring of polynomials with algebraic coefficients in countably many variables $X_\alpha$, $Y_\alpha$, $\alpha\in\Qbar$:
$$
\mathcal A:=\Qbar[(X_\alpha, Y_\alpha)_{\alpha\in\Qbar}] \,.
$$
We recall that each element of $\mathcal A$ is a polynomial with algebraic coefficients in only finitely many of the variables $X_\alpha$ and $Y_\alpha$. 
The ring $\mathcal A$ is also a $\Qbar$-algebra. Let $\ev$ denote the evaluation map from $\mathcal A$ to $\mathcal T$ defined by 
$$
\ev(X_\alpha)=\cos(\alpha)\quad\textup{and}\quad \ev(Y_\alpha)=\sin(\alpha)\,,  \quad \alpha\in\Qbar\,.
$$
This defines a surjective (but non-injective) homomorphism of $\Qbar$-algebras from $\mathcal A$ to $\mathcal T$. Let $\kI$ denote the ideal of $\mathcal A$ spanned by the polynomials 
\begin{equation}\label{eq: I generators}
X_\alpha^2+Y_\alpha^2-1,\quad X_{\alpha+\beta}-X_\alpha X_\beta +Y_\alpha Y_\beta\,, \quad\textup{and}\quad Y_{\alpha+\beta}-Y_\alpha X_\beta-X_\alpha Y_\beta\,,
\end{equation}
where $\alpha$ and $\beta$ run along $\Qbar$.  The identities \eqref{eq: pyt} and \eqref{eq: add} imply that $\ev(\kI)=\{0\}$. Hence 
the map $\ev$ allows us to define a 
homomorphism of $\Qbar$-algebras $\overline{\ev}$ from the quotient algebra $\mathcal A/\kI$ to $\mathcal T$. Now, our vague theorem can be properly formalized as follows. 

\begin{theorem}\label{theo: main}
The map $\overline{\ev}: \mathcal A/\kI\rightarrow \mathcal T$ is an isomorphism. 
\end{theorem}

Although it is deduced quite directly from the famous Lindemann--Weierstrass theorem using some basic notions of commutative algebra, we were unable to find this result in the literature. We felt this statement was interesting enough for a broad audience of mathematicians to write this note. We will also draw a parallel with some of the deeper conjectures and results that permeate transcendental number theory at the end of this note.

\section{Proof of Theorem 1.}


\subsection{Reflection formulas and angle difference identities.}

Beyond the Pythagorean identity and the angle addition formulas, there are other classical algebraic relations between the values of sine and cosine that come to mind. For example, one can think of the initial values $\sin(0)=0$ and $\cos(0)=1$, the reflection formulas $\cos(-\alpha)=\cos(\alpha)$ and $\sin(-\alpha)=-\sin(\alpha)$, and  the angle difference identities 
$$
\begin{cases}
\cos(\alpha-\beta)=\cos(\alpha)\cos(\beta)+\sin(\alpha)\sin(\beta)\,,\\ 
\sin(\alpha-\beta)=\sin(\alpha)\cos(\beta)-\cos(\alpha)\sin(\beta)\,.
\end{cases}
$$
The following lemma shows that all of them belong to $\kI$.

\begin{lem}\label{lem: parity}
For every $\alpha$ and $\beta$ in $\Qbar$, the polynomials 
$$
\begin{cases}
X_{\alpha} -  X_{-\alpha}\,,\\
Y_{\alpha}+Y_{-\alpha}\,,\\ 
X_{\alpha-\beta}-X_\alpha X_\beta-Y_\alpha Y_\beta\,,\\
Y_{\alpha-\beta}-Y_\alpha X_\beta+X_\alpha Y_\beta\,, 
\end{cases}
$$
as well as the polynomials $Y_0$ and $X_0-1$, 
all belong to $\kI$. 
\end{lem}

\begin{proof}
Throughout this proof, the congruences hold modulo $\kI$. We first infer from \eqref{eq: I generators} with $\alpha=\beta=0$ that 
\begin{equation}\label{eq: ref eq}
X_0^2+Y_0^2\equiv 1,\quad X_0\equiv X_0^2-Y_0^2\quad\textup{and}\quad Y_0\equiv 2X_0Y_0\,.
\end{equation}
We obtain the system
\begin{equation}\label{eq: system X_0 Y_0}
\left(\begin{array}{cc}
  X_0 & Y_0 \\
  X_0 & -Y_0  
\end{array}\right)
\left(\begin{array}{c}
X_0 \\
Y_0
\end{array}\right)
\equiv
\left(\begin{array}{c}
1\\ X_0
\end{array}\right),
\end{equation} 
with determinant $-2X_0 Y_0\equiv -Y_0$. Left multiplying System \eqref{eq: system X_0 Y_0} by its adjugate matrix, we obtain that
$$
-Y_0\left(\begin{array}{c}
X_0\\ Y_0
\end{array}\right)
\equiv
\left(\begin{array}{cc}
  -Y_0 & -Y_0 \\
  -X_0 & X_0  
\end{array}\right)
\left(\begin{array}{c}
1 \\
X_0
\end{array}\right)
\equiv
\left(\begin{array}{c}
-Y_0-Y_0X_0 \\
-X_0+X_0^2
\end{array}\right).
$$ 
 As a consequence, we obtain that $Y_0\equiv 0$. By \eqref{eq: ref eq}, it follows that $X_0^2\equiv 1$ and $X_0\equiv X_0^2\equiv 1$. Hence both $Y_0$ and $X_0-1$ belong to $\kI$. 
 
 Let $\alpha\in \Qbar$. Using the angle addition formulas with $\beta=-\alpha$, we obtain that
$$
1\equiv X_\alpha X_{-\alpha}-Y_\alpha Y_{-\alpha}\quad\textup{and}\quad 0\equiv Y_\alpha X_{-\alpha}+X_\alpha Y_{-\alpha}\,, 
$$
that is
$$
\left(\begin{array}{cc}
    X_{-\alpha} & -Y_{-\alpha} \\
    Y_{-\alpha} & X_{-\alpha}
\end{array}\right)\left(\begin{array}{c}
     X_\alpha  \\
      Y_\alpha
\end{array}\right)\equiv\left(\begin{array}{c}
     1  \\
      0
\end{array}\right)\,.
$$
By the Pythagorean identity, this linear system is invertible and we have
$$
\left(\begin{array}{c}
     X_\alpha  \\
      Y_\alpha
\end{array}\right)\equiv 
\left(\begin{array}{cc}
    X_{-\alpha} & Y_{-\alpha} \\
    -Y_{-\alpha} & X_{-\alpha}
\end{array}\right)
\left(\begin{array}{c}
     1  \\
      0
\end{array}\right)\equiv
\left(\begin{array}{c}
     X_{-\alpha}  \\
      -Y_{-\alpha}
\end{array}\right),
$$
which gives that $X_\alpha-X_{-\alpha}$ and $Y_{\alpha}+Y_{-\alpha}$ belong to $\kI$. Then, using the angle addition formulas, we deduce that the polynomials 
$$X_{\alpha-\beta}-X_\alpha X_\beta-Y_\alpha Y_\beta \quad \mbox{ and } \quad 
Y_{\alpha-\beta}-Y_\alpha X_\beta+X_\alpha Y_\beta\,, \quad\quad \alpha,\beta\in\Qbar\,,$$
also belong to $\kI$, as expected. 
\end{proof}

\subsection{The Lindemann--Weierstrass theorem and basic commutative algebra.}

One of the gems of transcendental number theory is the Lindemann-Weierstrass theorem.

\begin{theorem}[Lindemann-Weierstrass]
Let $\alpha_1,\ldots,\alpha_n$ be algebraic numbers that are linearly independent over $\mathbb Q$. Then the complex numbers 
$e^{\alpha_1},\ldots,e^{\alpha_n}$ are algebraically independent over $\Qbar$. 
\end{theorem}

We will first deduce from the Lindemann--Weierstrass theorem the following result. 
Given a tuple of complex numbers $\cE:=(\zeta_1,\dots,\zeta_m)$, we let
\begin{equation}\label{eq: alge}
\Alg(\cE):=\{P(X_1,\dots,X_m)\in\Qbar[X_1,\dots,X_m]\,:\,P(\zeta_1,\dots,\zeta_m)=0\}
\end{equation}
denote the ideal of algebraic relations over $\Qbar$ between the coordinates of $\cE$.

\begin{cor}\label{cor: lin ind case}
Let $\alpha_1,\dots,\alpha_n$ be $\bQ$-linearly independent algebraic numbers and set 
$$\cE:=(\cos(\alpha_1),\sin(\alpha_1),\dots,\cos(\alpha_n),\sin(\alpha_n)) \,.$$ Then $\{X_i^2+Y_i^2-1 : 1\leq i\leq n\}$ is a set of generators of $\Alg(\cE)$, viewed as an ideal  
of $\Qbar[X_1,Y_1,\dots,X_n,Y_n]$.
\end{cor}

This corollary will probably come as no surprise to specialists and may seem obvious to them, but as this note is intended for a wide audience 
we have chosen to give a detailed proof. To this end, we begin by recalling some basic notions of commutative algebra, for which we refer the reader to 
\cite[Chap.\ 8]{Bo}. 

Let $R$ be a unitary commutative ring. We say that a chain of prime ideals of $R$ of the form ${\displaystyle {\mathfrak {p}}_{0}\subsetneq {\mathfrak {p}}_{1}\subsetneq \ldots \subsetneq {\mathfrak {p}}_{n}}$ has length $n$. The Krull dimension of $R$, denoted by $\dim R$,   is defined as the supremum of the lengths of all chains of prime ideals. For example, if $\mathbb K$ is a field, then $\dim \mathbb K[X_1,\ldots,X_n]=n$. Given a prime ideal $\mathfrak p$ of $R$, the height of 
$\mathfrak p$,  denoted by $\rmht(\mathfrak p)$, is defined as the supremum of the lengths of all chains of prime ideals contained in 
$\mathfrak p$. 

Let $\mathbb K$ be a field and $\mathbb L$ be a field extension of $\mathbb K$.  A subset $S$ of $\mathbb L$ is said to be algebraically independent 
over $\mathbb K$, if the elements of $S$ do not satisfy any non-trivial polynomial equation with coefficients in $\mathbb K$. 
The transcendence degree of $\mathbb L$ over $\mathbb K$, denoted by $ {\rm tr.deg}_{\mathbb K}\mathbb L$, is then defined as the 
maximal cardinality among all algebraically independent subsets  $S\subseteq \mathbb L$. 

Let $A$ be a finitely generated integral $\mathbb K$-algebra, say $A:=\mathbb K[a_1,\ldots,a_r]$. Then, the Noether normalization lemma implies that 
\begin{equation}\label{eq: dim}
\dim A = {\rm tr.deg}_{\mathbb K}\mathbb K(a_1,\ldots,a_r) \leq r \,.
\end{equation}
Furthermore, if $\mathfrak p$ is a prime ideal of $A$ then 
\begin{equation}\label{eq: dimq}
\dim A/\mathfrak p = \dim A -\rmht(\mathfrak p)\,. 
\end{equation}

\begin{proof}[Proof of Corollary \ref{cor: lin ind case}]
By Euler's formula, we have  
$$
e^{i\alpha}= \cos(\alpha) +i\sin(\alpha)\,,
$$
for all complex numbers $\alpha$. It follows that the complex numbers $e^{i\alpha_1},\dots,e^{i\alpha_n}$ all belong to 
$$\Qbar[\mathcal E]=\Qbar[\cos(\alpha_1),\sin(\alpha_1),\dots,\cos(\alpha_n),\sin(\alpha_n)]\,.$$
Since the numbers $i\alpha_1,\dots,i\alpha_n$ are linearly independent over $\bQ$, the Lindemann--Weierstrass theorem implies that the numbers $e^{i\alpha_1},\dots,e^{i\alpha_n}$ are algebraically independent over $\Qbar$. We deduce that 
\begin{equation}\label{eq: tre}
 {\rm tr.deg}_{\Qbar}\mathbb \Qbar(\mathcal E)  \geq n \,.
\end{equation}

The $\Qbar$-algebra $\Qbar[\cE]$  is isomorphic to
$
\Qbar[X_1,Y_1,\dots,X_n,Y_n]/\Alg(\cE)$. 
Since $\Qbar[\cE]$ is a finitely generated integral $\Qbar$-algebra,  Equalities \eqref{eq: dim} and \eqref{eq: dimq} yield 
\begin{align*}
    \degtr_{\Qbar}\Qbar(\cE) &= \dim\Qbar[\cE]\\
    &= \dim \Qbar[X_1,Y_1,\dots,X_n,Y_n]/\Alg(\cE)\\
    &= 2n - \rmht(\Alg(\cE))\,,
\end{align*}
where the last equality holds because $\Alg(\cE)$ is a prime ideal. It follows from \eqref{eq: tre} that 
\begin{equation}\label{eq: height bound}
\rmht(\Alg(\cE))\leq n\,.
\end{equation}

For every $k$, $1\leq k\leq n$, we let $\mathfrak p_k$ denote the ideal of $\Qbar[X_1,Y_1,\dots,X_n,Y_n]$ spanned by the polynomials $X_i^2+Y_i^2-1$, 
$1\leq i\leq k$.  Thus, we want to prove that  $\mathfrak p_n= \Alg(\cE)$. 
By the Pythagorean identity \eqref{eq: pyt}, we have the ascending chain
$$
\{0\}\varsubsetneq \mathfrak p_1\varsubsetneq\cdots\varsubsetneq \mathfrak p_n\subseteq \Alg(\cE)\,.
$$
By \eqref{eq: height bound}, if all these ideals are prime, we obtain that $\mathfrak p_n=\Alg(\cE)$, as wanted. 
It thus remains to prove that $\mathfrak p_k$ is a prime ideal for all $k\in\{1,\ldots,n\}$. 

To that purpose, we first prove the following claim: 

\medskip

\emph{if $R$ is an integral domain of characteristic $\not=2$, then $R[X,Y]/(X^2+Y^2-1)$ is also an integral domain.}

\medskip

Indeed, let $P,Q\in R[X,Y]$ be such that $PQ\in(X^2+Y^2-1)$. Substituting $Y^2$ by $1-X^2$ in $P$ and $Q$, we see that there exist polynomials $P_0,P_1,Q_0,Q_1 \in R[X]$ such that we have $P\equiv P_0 + P_1 Y$ and $Q\equiv Q_0+Q_1 Y$ modulo $(X^2+Y^2-1)$. It follows that
$$
0\equiv PQ\equiv P_0Q_0+P_1Q_1(1-X^2)+(P_0Q_1+Q_0P_1)Y\mod (X^2+Y^2-1)\,,
$$
which yields $P_0Q_0+P_1Q_1(1-X^2)=P_0Q_1+Q_0P_1=0$. We thus obtain the linear system
$$
\left(\begin{array}{cc}
    Q_0 & Q_1(1-X^2) \\
    Q_1 & Q_0
\end{array}\right)\left(\begin{array}{c}
     P_0  \\
      P_1
\end{array}\right)=0\,.
$$
Hence either $P_0=P_1=0$ and then $P\in(X^2+Y^2-1)$, or $Q_0^2-Q_1^2(1-X^2)=0$. In the latter case, we thus have 
$$
Q_0^2=Q_1^2(1-X^2)\,, 
$$
but, since the characteristic of $R$ is not equal to $2$, the valuation at $(1-X)$ is even on the left-hand side and odd on the right-hand side, 
unless $Q_0=Q_1=0$. 
We deduce that $Q_0=Q_1=0$ and hence $Q\in (X^2+Y^2-1)$. 
This proves our claim.

Let $R$ be an integral domain and $\mathfrak p$ be an ideal of $R$.
Then the extended ideal $\mathfrak p^e$ in $R[X_1,\ldots,X_r]$ is defined as the ideal of $R[X_1,\ldots,X_r]$ 
generated by the elements of $\mathfrak p$, that is 
the set of polynomials in the variables $X_1,\ldots,X_r$ whose coefficients belong to $\mathfrak p$. This definition implies that 
\begin{equation}\label{eq: ideale}
\left(R/\mathfrak p\right)[X_1,\ldots,X_r] \cong R[X_1,\ldots,X_r]/\mathfrak p^e\,.
\end{equation}
If $\mathfrak p$ is prime, we have that $R/\mathfrak p$ is integral and 
hence $(R/\mathfrak p)[X_1,\ldots,X_r]$  is integral too, and then by \eqref{eq: ideale},  $\mathfrak p^e$ is also prime. 
Furthermore, if $\mathfrak p$ and $\mathfrak q$ are two ideals of $R$, then 
\begin{equation}\label{eq: quot}
(R/\mathfrak p)/\overline{\mathfrak q}\cong R/(\mathfrak p+\mathfrak q)\,,
\end{equation}
where $\overline{\mathfrak q}$ is the image of the ideal $\mathfrak p + \mathfrak q$ by the natural projection from $R$ to $R/\mathfrak p$\footnote{We recall that there is a one-to-one correspondence between the ideals of 
$R/\mathfrak p$ and the ideals of $R$ that contain $\mathfrak p$.}. 

Now, let $\mathfrak q_k$ denote the ideal of $R_k:=\Qbar[X_1,Y_1,\ldots,X_k,Y_k]$ spanned by the polynomials $X_i^2+Y_i^2-1$, $1\leq i\leq k$. Let us 
prove that these ideals are all prime by induction on $k$.  
Using the claim, we obtain that $R_1/\mathfrak q_1$ is an integral domain of characteristic zero and hence $\mathfrak q_1$ is prime. 
Now, let $k<n$ and let us assume that $\mathfrak q_k$ is prime. 
We infer from \eqref{eq: ideale} and \eqref{eq: quot} that 
\begin{align*}
(R_k/\mathfrak q_k)[X_{k+1},Y_{k+1}] /(X_{k+1}^2+Y_{k+1}^2-1)  &\cong \left(R_{k+1}/\mathfrak q_k^e\right)/(X_{k+1}^2+Y_{k+1}^2-1)\\ 
&\cong  R_{k+1}/\mathfrak q_{k+1}\,.
\end{align*}
Since $\mathfrak q_k$ is prime, we get that $R_k/\mathfrak q_k$ is an integral domain of characteristic zero, and the claim ensures that the left-hand side is an integral domain. We deduce that $R_{k+1}/\mathfrak q_{k+1}$  is an integral domain and hence 
$\mathfrak q_{k+1}$ is prime.  Thus, all the $\mathfrak q_k$ are prime. 
Since $\mathfrak p_k$ is the extended ideal of $\mathfrak q_k$ in $\mathbb K[X_1,Y_1,\ldots,X_n,Y_n]$, it follows that 
all the $\mathfrak p_k$ are prime, as desired.  
\end{proof}

\subsection{Proof of Theorem \ref{theo: main}.}

We are now ready to deduce our main result. 

\begin{proof}[Proof of Theorem \ref{theo: main}] 
The map $\ev$ is clearly surjective, so we only have to prove that $\ker(\ev) =\kI$. 
We have already observed that $\kI\subseteq\ker(\ev)$, so now we just have to prove the reverse inclusion. Let $P\in \mathcal A$ be such that $P\in\ker(\ev)$ and let  us prove that $P\in\kI$.   

We define the support of a polynomial $Q\in\mathcal A=\Qbar[X_\alpha,Y_\alpha : \alpha\in\Qbar]$ 
as the smallest set $\mathcal S_1\times \mathcal S_2\subseteq\Qbar\times\Qbar$ such that $Q$ can be written as a polynomial in the variables $X_\alpha$,  $\alpha \in \mathcal S_1,$ and $Y_\beta$, $\beta\in\mathcal S_2$, meaning that the support is the intersection of all such finite sets.  
We let $S:=\{\alpha_1,\ldots,\alpha_n\} \subseteq\Qbar$ be a finite set such that the support of $P$ is included in $\mathcal S\times \mathcal S$. 

Let $k$ denote the dimension of the $\mathbb Q$-vector space $V$ generated by the elements of $\mathcal S$. 
Reordering if necessary, we can assume without any loss of generality that $\alpha_1,\dots,\alpha_k$ is a basis of $V$.  Then there exists a positive integer $d$ such that  
all the elements of $\mathcal S$ belong to the $\bZ$-module spanned by the numbers $\alpha_i/d$, $1\leq i\leq k$. By Lemma \ref{lem: parity}, we have the following relations modulo $\kI$: 
$$
\left\{\begin{array}{ccccc}
X_{\alpha+\beta} & \equiv & X_\alpha X_\beta -Y_\alpha Y_\beta\, ,\\
Y_{\alpha+\beta} & \equiv & Y_\alpha X_\beta+X_\alpha Y_\beta\, ,\\
X_{\alpha-\beta} & \equiv & X_\alpha X_\beta +Y_\alpha Y_\beta\,,\\
Y_{\alpha-\beta} & \equiv & Y_\alpha X_\beta-X_\alpha Y_\beta\,.
\end{array}\right.
$$
 These congruences show that, for every $\gamma$ in $S$, there exist polynomials $P_\gamma$ and $Q_\gamma$ whose support is included in  
 $\{\alpha_1/d,\dots,\alpha_k/d\}^2$ and such that $X_\gamma\equiv P_\gamma$ and $Y_\gamma\equiv Q_\gamma$ modulo $\kI$. It follows that there exists a polynomial $U$ with support in $\{\alpha_1/d,\dots,\alpha_k/d\}^2$ such that  $P\equiv U\mod \kI$. Since $P\in\ker(\ev)$ and $\kI\subseteq\ker(\ev)$, 
 it follows that $U\in\ker(\ev)$ too.
 Since the numbers $\alpha_i/d$, $1\leq i\leq k$, are $\bQ$-linearly independent, Corollary~\ref{cor: lin ind case} implies that $U$ belongs to the ideal 
 spanned by the polynomials $X_{\alpha_i/d}^2+Y_{\alpha_i/d}^2-1$, $1\leq i\leq k$. We deduce that $U\in\kI$ and hence $P\in\kI$, as desired.
\end{proof}

\section{Finding explicit relations.}\label{sec: algo}

Theorem \ref{theo: main} provides the \emph{raison d'\^{e}tre} of the algebraic relations over $\Qbar$ between the values of sine and cosine at algebraic points: they all come from the fundamental identities \eqref{eq: pyt} and \eqref{eq: add}. We now show that, given algebraic numbers $\alpha_1,\dots,\alpha_n$, they can be used to explicitly find generators of the ideal of the algebraic relations $\Alg(\cE)$ defined in \eqref{eq: alge}, where 
$$
\cE:=(\cos(\alpha_1),\sin(\alpha_1),\dots,\cos(\alpha_n),\sin(\alpha_n))\,.
$$
More precisely, we describe an algorithm to compute a set of generators of the ideal $\Alg(\cE)$. Observe that Theorem \ref{theo: main} does not imply that the ideal spanned by the polynomials \eqref{eq: I generators} with support in $\{\alpha_1,\dots,\alpha_n\}$ is equal to $\mathrm{Alg}_{\Qbar}(\cE)$. However, as described below, only a finite number of polynomials of the form \eqref{eq: I generators} are needed to span this ideal. 

Our input is a finite set of algebraic numbers  $\alpha_1,\dots,\alpha_n$. We assume that algebraic numbers are given by their minimal integer polynomial and a  rational approximation precise enough to separate it from the other roots of this polynomial. 

\subsubsection*{First step.} The data concerning the $\alpha_i$'s provide an explicit bound on their naive height (\emph{i.e.}, the maximum of the absolute values of the coefficients of their minimal integer polynomial). We can thus use Siegel lemma \cite{Si} to compute   
an explicit integer $M$ such that if there is a $\bQ$-linear relation between some of the $\alpha_i$'s, then one exists whose coefficients have height at most $M$. 
Then we can use  the algorithm given in \cite{Ju} to find a basis 
$(\beta_1,\dots,\beta_k)\subseteq \{\alpha_1,\ldots,\alpha_n\}$ of  the $\bQ$-vector space $\Span_\bQ(\alpha_1,\dots,\alpha_n)$, and, for each $i\in \{1,\ldots, n\}$, 
an expression of $\alpha_i$ as  
a $\bQ$-linear combination of the $\beta_j$, $1\leq j\leq k$. Let $d$ be the $\lcm$ of the denominator of the coefficients of these linear combinations, then 
we can explicitly find rational integers $a_{i,j}$ such that 
$$
\alpha_i=a_{i,1}\frac{\beta_1}{d}+\cdots+a_{i,k}\frac{\beta_k}{d}\,,\quad 1\leq i\leq n\,.
$$

\subsubsection*{Second step.}
Let us fix an integer $i$ in $\{1,\dots,n\}$. Assume that $a_{i,1}>0$. Then writing $\alpha_i = (\alpha_i - \beta_1/d) + \beta_1/d$ and using the angle addition formulas, one can write both 
$\cos(\alpha_i)$ and $\sin(\alpha_i)$ as polynomials in $\cos(\alpha_i -\beta_1/d)$, $\sin(\alpha_i -\beta_1/d)$, $\cos(\beta_1/d)$, and $\sin(\beta_1/d)$. 
Using this procedure $a_{i,1}-1$ times, we end up with 
expressions of $\cos(\alpha_i)$ and $\sin(\alpha_i)$ as polynomials in $\cos(a_{i,2}\frac{\beta_1}{d}+\cdots+a_{i,k}\frac{\beta_k}{d})$,  $\sin(a_{i,2}\frac{\beta_1}{d}+\cdots+a_{i,k}\frac{\beta_k}{d})$,  $\cos(\beta_1/d)$, and $\sin(\beta_1/d)$. 
If $a_{i,1}$ is negative, one just replaces the angle addition formulas by the angle difference formulas. 
Then one can use the same procedure with $a_{i,2}$ and so on. 
In the end, we can explicitly find polynomials $P_i$ and $Q_i$ in $\mathbb{Z}[W_1,Z_1,\dots,W_k,Z_k]$, $1\leq i\leq n$, such that, for every $i\in\{1,\dots,n\}$, one has
$$
\cos(\alpha_i)=P_i\left(\cos\left(\frac{\beta_1}{d}\right),\sin\left(\frac{\beta_1}{d}\right),\dots,\cos\left(\frac{\beta_k}{d}\right),\sin\left(\frac{\beta_k}{d}\right)\right)
$$
and
$$
\sin(\alpha_i)=Q_i\left(\cos\left(\frac{\beta_1}{d}\right),\sin\left(\frac{\beta_1}{d}\right),\dots,\cos\left(\frac{\beta_k}{d}\right),\sin\left(\frac{\beta_k}{d}\right)\right)\,.
$$

\subsubsection*{Third step.} At this point, the polynomials $$P_i, Q_i\in\mathbb Z[W_1,Z_1,\ldots,W_k,Z_k]\,, 
\quad  1\leq i\leq n\, ,$$ have been computed. 
Now consider the tuple $\cF$ of complex numbers formed by the concatenation of $\cE$ and
$$
\left(\cos(\beta_1/d),\sin(\beta_1/d),\dots,\cos(\beta_k/d),\sin(\beta_k/d)\right)\,.
$$
We claim that the ideal $\Alg(\cF)$ of $\Qbar[X_1,Y_1,\dots,X_n,Y_n,W_1,Z_1,\dots,W_k,Z_k]$ is generated by the polynomials $X_i-P_i$, $Y_i-Q_i$, $1\leq i\leq n$, and $W_j^2+Z_j^2-1$, $1\leq j\leq k$. Let $\mathfrak p$ denote the ideal generated by these polynomials. Clearly, $\mathfrak p\subseteq \Alg(\cF)$. 
Let $P\in\Alg(\cF)$.  Since $X_i\equiv P_i$ and $Y_i\equiv Q_i$ modulo $\mathfrak p$, there exists a polynomial $R\in\Qbar[W_1,Z_1,\dots,W_k,Z_k]$ such that $P\equiv R\mod \mathfrak p$. Since $\beta_1/d,\dots,\beta_k/d$ are $\bQ$-linearly independent, Corollary~\ref{cor: lin ind case} implies that $R$ belongs to the ideal spanned by the polynomials $W_j^2+Z_j^2-1$, $1\leq j\leq k$. It follows that $R\in \mathfrak p$ and hence $P\in \mathfrak p$. We deduce that $\Alg(\cF)=\mathfrak p$, as claimed. 

The polynomials $X_i-P_i$, $Y_i-Q_i$, $1\leq i\leq n$, and $W_j^2+Z_j^2-1$, $1\leq j\leq k$ 
form a set of generators of  $\Alg(\cF)$, so we can consider a monomial order eliminating the variables $W_j$ and $Z_j$, $1\leq j\leq k$, and 
use Buchberger's algorithm to find an associated Gr\"{o}bner basis $\cG$ of $\Alg(\cF)$, and hence  
$\cG\cap\Qbar[X_1,Y_1,\dots,X_n,Y_n]$ is a Gr\"obner basis of  $\Alg(\cE)$ (see, for example, \cite{BW}). 
Then the output of our algorithm is simply $\cG\cap\Qbar[X_1,Y_1,\dots,X_n,Y_n]$, which provides an explicit set of generators of the ideal  $\Alg(\cE)$, as wanted.
 
\section{Concluding remarks.}

In this final section, we place our results in a broader context, with a few references to more advanced work that motivated the writing of this note.

Transcendental number theory has two main driving forces. Given a set of complex numbers $\Xi$, one wishes, on the one hand, to be able to determine, 
for all $\xi_1,\ldots,\xi_r\in\Xi$, the algebraic relations over $\Qbar$ between these numbers, and, on the other, to find the 
\emph{raison d'\^{e}tre} for these putative relations. These two problems are naturally linked, and the second, more ambiguous, depends of course on how the elements of $\Xi$ are defined. To make our point a little clearer, we begin by briefly recalling three famous conjectures. Solving the first two would answer this second problem in the case of the ring of periods,\footnote{We recall that a period is a complex number whose real and imaginary parts 
are values of absolutely convergent integrals of rational functions with rational coefficients, over domains in $\mathbb R^n$ given by polynomial inequalities with rational coefficients. Many classical mathematical constants such as the algebraic numbers, $\pi$, $\log 2$, and $\zeta(3)$ are periods. } while solving the third  would answer it in the case of (normalized) values of the Euler Gamma function at rational points. 

In the case where $\Xi$ denotes the ring of periods, a conjecture due to 
Kontsevich and Zagier \cite{KZ} predicts that any algebraic relation 
between periods can be derived from the fundamental rules of integral calculus--additivity, change of variables and the Stokes formula.  
We point out to the interested reader that this conjecture is essentially equivalent to the famous Grothendieck period conjecture, although the latter is expressed in terms too elaborate to be defined 
here.\footnote{It involves a more geometric definition of periods associated with algebraic varieties, their (Betti) homology and de Rham cohomology, 
and their motivic Galois group.} An in-depth discussion of these two conjectures and their links can be found in  \cite{Fre22}.

In the second case, we choose $$\Xi:=\left\{\frac{\Gamma(r)}{\sqrt{2\pi}} : r\in\mathbb Q\right\}\,,$$ where $\Gamma$ is the Euler gamma function. Then the Rohrlich--Lang conjecture predicts that all algebraic relations come from specializations of standard functional relations associated with $\Gamma$ (see, for example, \cite[Conjecture 22]{Miw}). 

In addition to sharing a common ambition, these three famous conjectures have in common the fact that they are considered to be totally beyond the reach of current methods.  It is therefore remarkable that the two main objectives we just described  have recently been achieved in a few cases. Namely, when $\Xi$ is: 

\begin{itemize}
\item[](a) the set of values at algebraic points of Siegel $E$-functions (see \cite{AF24,FR19}); 
\item[](b) the set of values at algebraic points of Mahler $M$-functions (see \cite{AF24,AF20}).
\end{itemize}

A Siegel $E$-function is defined as a power series of the form 
$$
f(z)=\sum_{n=0}^\infty \frac{a_n}{n!} z^n \,,
$$
where $(a_n)_{n\geq 0}$ is a sequence of algebraic numbers satisfying some arithmetic growth condition and such that $f(z)$ satisfies a linear differential equation with coefficients in $\Qbar[z]$ (see \cite{Si} for a precise definition). These functions can be seen as generalizations of the exponential function, which corresponds to the case where $(a_n)_{n\geq 0}$ is the constant function equal to $1$. The Siegel-Shidlovskii theorem (see \cite{Sh_Liv}) is a fundamental result concerning the algebraic relations between values of $E$-functions at nonzero algebraic points that provides a broad generalization of the Lindemann-Weierstrass theorem.  
An important refinement of the Siegel-Shidlovskii theorem was obtained by Beukers in \cite{Be06} and 
it was recently observed in \cite{AF24} that it can be used to prove that all algebraic relations over $\Qbar$ between values of $E$-functions at nonzero algebraic points have a \emph{functional} origin. More concretely, if  $f_1(z),\ldots,f_r(z)$ are $E$-functions whose values at a nonzero algebraic point 
$\alpha$ satisfy a polynomial relation with algebraic coefficients, then this relation can be obtained as the specialization at $\alpha$ of a polynomial relation with coefficients in $\Qbar[z]$ between the functions $f_1(z),\ldots,f_r(z)$ and a finite number of their successive derivatives. 
In this context, this result provides a satisfactory answer to the second problem, while it can also be used to solve the first one as also explained in \cite{AF24}. Indeed, it allows us to prove that given a finite number of values (at nonzero algebraic points) of $E$-functions, there exists an algorithm to find an explicit basis of the ideal of the algebraic relations over 
$\Qbar$ between these numbers. 
The existence of such an algorithm was first proved in \cite{FR19} in a slightly different way. 

A Mahler $M$-function is defined as a power series 
$$
f(z)=\sum_{n\geq 0}a_nz^n \,,
$$
where $a_n\in\Qbar$ and such that $f(z)$ satisfies a linear difference equation with coefficients in $\Qbar[z]$ of the form 
$$
a_0(z)f(z)+a_1(z)f(z^q)+\cdots+a_m(z)f(z^{q^m})=0\,,
$$ 
where $q\geq 2$ is an integer.  
The results obtained in \cite{AF24,AF20} in the case of Mahler $M$-functions are similar to those obtained in the case of $E$-functions, the role played by the derivative $\frac{d}{dz}$ being replaced by the Mahler operators $\sigma_q: z\mapsto z^q$.

This note was inspired by these recent  results. Indeed, Theorem~\ref{theo: main} and the result in Section~\ref{sec: algo} solve the two problems in the case where 
$\Xi=\left\{\cos(\alpha),\sin(\alpha) : \alpha\in\Qbar\right\}$. Since the functions sine and cosine are $E$-functions,  this is a special case of (a), but one for which we can provide an elementary proof and also a concrete and beautiful description of the algebraic relations.  In the same vein, a similar result could be obtained for the set $\Xi=\{J_\lambda(\alpha) : \alpha\in\Qbar,\, \lambda \in\mathbb Q\setminus \mathbb Z_{< 0}\}$ using the theorem proved by Siegel in his famous memoir \cite{Si} concerning the algebraic independence of values at algebraic points of the Bessel functions 
$$J_\lambda(z)=\sum_{n=0}^\infty \frac{(-1)^n}{n! \Gamma(n+\lambda+1)} \left(\frac{z}{2}\right)^{2n+\lambda}\,.$$


\end{document}